\documentclass[12pt, a4paper]{article}

\usepackage{amsfonts, latexsym, amssymb, amsthm, amsmath, mathrsfs, esint, color}
\usepackage[latin1]{inputenc}
\usepackage{enumerate}
\usepackage[active]{srcltx}

\usepackage{hyperref}

\renewcommand{\a}{\alpha}
\renewcommand{\b}{\beta}
\newcommand{\D}{\nabla}

\newcommand{\h}{\widehat}
\renewcommand{\O}{\Omega}
\def\e{\varepsilon}

\renewcommand{\t}{\widetilde}
\newcommand{\R}{\mathbb{R}}

\def\N{\mathbb{N}}

\def\D{\nabla}
\def\a{\alpha}
\def\b{\beta}

\def\XXint#1#2#3{{\setbox0=\hbox{$#1{#2#3}{\int}$}
     \vcenter{\hbox{$#2#3$}}\kern-.5\wd0}}

\newcounter{bei}

\renewcommand{\o}{\overline}

\newcommand{\zwo}[2]{\begin{pmatrix} {#1}\\{#2} \end{pmatrix}}

\newtheorem{theorem}{Theorem}[section]

\newtheorem{proposition}[theorem]{Proposition}
\newtheorem{corollary}[theorem]{Corollary}

\DeclareMathOperator{\curl}{curl}

\DeclareMathOperator{\dist}{dist}

\DeclareMathOperator{\sym}{sym}

\begin{document}

\title{$C^{1,\a}$ $h$-principle for
\\
von K\'arm\'an constraints}

\author{{\sc Jean-Paul Daniel}
and {\sc Peter Hornung}\footnote{Address correspondence to: Fachbereich Mathematik,
TU Dresden, Germany, {\tt peter.hornung@tu-dresden.de}.}
}

\date{}

\maketitle

\begin{abstract}
Exploiting some connections between solutions $v : \Omega\subset\R^2\to\R$, 
$w : \Omega\to\R^2$ of
the system $\D v\otimes\D v + 2\sym\D w = A$
and the isometric immersion problem in two dimensions,
we provide a simple construction of $C^{1,\a}$ convex integration solutions for the former
from the corresponding result for the latter.
\end{abstract}



\section{Introduction and main result}

The classical $h$-principle of Nash and Kuiper shows that there exist
surprisingly many $C^1$ solutions 
$u : \Omega\subset\R^2 \to\R^3$
to the isometric immersion system
\begin{equation}
\label{intro-1}
(\D u)^T(\D u) = g.
\end{equation}
In contrast, classical rigidity
results show that, among more regular immersions, being a solution of system \eqref{intro-1}
is as restrictive a condition as one might expect. A natural question
is whether such results extend to classes of $C^{1,\a}$ H\"older spaces:
for the $h$-principle one seeks the largest possible H\"older exponent
$\a\in (0, 1)$ and for the rigidity
the smallest possible one. We refer to \cite{CDS} and the references therein.
\\
Such a dichotomy between an $h$-principle on one hand and rigidity on the other
hand also applies to other PDE systems. A system for which this is to be expected
is the system
\begin{equation}
\label{intro-2}
\D v\otimes\D v + 2\sym\D w = A
\end{equation}
for $w : \Omega\to\R^2$ and $v : \Omega\to\R$.
This system arises naturally as a constraint in von K\'arm\'an theories (cf. \cite{FJM2})
in certain energy regimes. In that context, $w$
describes the in-plane displacement and $v$
the out-of-plane displacement.
It is clearly related to the Monge-Amp\`ere equation 
$\det\D^2 v = \curl\curl A$. We refer e.g. to \cite{FJM2}
for some details on this.
\\
System \eqref{intro-2} is closely related to \eqref{intro-1},
and it was shown in \cite{LP} that the
convex integration construction in \cite{CDS} can indeed be adapted to obtain
the same statement for system \eqref{intro-2}.
In this paper we show how the close
connection between \eqref{intro-2} and \eqref{intro-1}
can be used to derive $C^{1,\a}$ $h$-principles for \eqref{intro-2} directly
from similar results for \eqref{intro-1}, without having to repeat the construction.
\\

From now on
$\Omega \subset \R^2$ denotes a bounded and simply connected
domain with a smooth boundary. Our main result is the following:
\begin{theorem}\label{thm1}
Let $p\geq 2$, let $\beta \in (0,1)$ and let
$$
0 < \alpha < \min \left\{ \frac{1}{7}, \frac{\beta}{2}  \right\}.
$$
Then there exists $C > 0$ such that the following is true:
\\
For any $v\in C^2(\overline \Omega)$, $w\in C^2(\o\O, \R^2)$,
$A\in C^{0,\beta}(\overline \Omega,  \R^{2 \times 2}_{\sym})$
there exist $\o v \in C^{1,\alpha}(\overline \Omega)$
and $\o w \in C^{1,\a}(\o\O, \R^2)$
with
\begin{equation}
\label{thm1-1}
\|\D \o v - \D v \|_{C^0(\Omega)} \leq C \|\D v\otimes\D v + 2\sym\D w - A\|_{C^0(\Omega)}^{1/2}
\end{equation}
and
\begin{equation}
\label{thm1-1}
\begin{split}
\|\D \o w - \D w\|_{L^p(\Omega)} &\leq C \|\D v\|_{C^0(\Omega)} 
\|\D v\otimes\D v + 2\sym\D w - A\|_{C^0(\Omega)}^{1/2}
\\
&+
\|\D v\otimes\D v + 2\sym\D w - A\|_{C^0(\Omega)},
\end{split}
\end{equation}
and such that
\begin{equation}
\label{thm1-22}
2\sym\D\o w + \D\o v\otimes\D \o v = A.
\end{equation}
\end{theorem}

{\bf Remarks.}
\begin{enumerate}
\item Theorem \ref{thm1} is a variant of \cite[Theorem 1]{CDS}.
It allows to improve a $C^1$ $h$-principle to
a $C^{1,\a}$ $h$-principle, cf. Corollary \ref{korollar}.
\item A variant of Theorem \ref{thm1} was stated 
in \cite{LP}. Theorem \ref{thm1} is more general
in that does not require $2\sym\D w + \D v\otimes\D v$
to be close to $A$. On the other hand 
it only yields $L^p$ rather than uniform bounds on $\D w$.
(For the actual convex integration result, however, this is immaterial. See
Corollary \ref{korollar} below.) The main difference to \cite{LP}
is that our short proof derives Theorem \ref{thm1} directly from the
corresponding result for isometric immersions \cite{CDS},
therefore avoiding the need of adapting each step of the construction in \cite{CDS}.
\end{enumerate}

\paragraph{Notation.}

For $n\in \N$, we denote  by $\mathbb{R}^{n \times n}_{\sym}$ the set of symmetric
$n \times n$ matrices. By $e$ we denote the standard Riemannian metric on $\R^n$.
Given an immersion $u$ into $\R^n$, we denote by $u^* e$ 
the pullback-metric, so that in coordinates
\begin{equation*}
(u^* e)_{ij}= \partial_i u \cdot \partial_j u.
\end{equation*}
For $k = 0, 1, ...$ we denote the usual $C^k$ norm by $\|u\|_k$.
For $\b\in (0, 1)$ the H\"older seminorm $[u]_{\b}$ is defined to
be the infimum over all $C$ such that
$$
|u(x) - u(y)|\leq C|x - y|^{\b}\mbox{ for all }x, y\in \O.
$$
The unit matrix is denoted by $I$.

\section{$h$-principle for isometric immersions}

An inspection of the proof in \cite{CDS} 
shows that in that paper the following more detailed version of \cite[Theorem 1]{CDS} is proven:

\begin{proposition}\label{haupt}
Let $n\in\N$, $\beta\in (0, 1)$, 
$$
0 < \a < \min\Big\{\frac{1}{1 + n(n+1)}, \frac{\beta}{2}\Big\},
$$
let  $g_0\in\R^{n\times n}_{\sym}$ be positive definite and let
$U\subset\R^n$ a smoothly bounded domain. There exist $\e_0$, $C$, $r > 0$
such that for all $\theta$, $\mu$, $\delta\in (0, \infty)$ satisfying
$\mu\geq\delta$ and
$
\delta^{\beta - 2}\mu^{-\beta}\theta^2 \leq \e_0
$
the following holds:
\\
If $g\in C^{0, \beta}(\o U, \R^{n\times n}_{\sym})$
satisfies
\begin{align}
\label{hr}
\|g - g_0\|_0 \leq r
\\
\label{htheta}
[g]_{\beta} \leq \theta^2
\end{align}
and if $u\in C^2(\o U, \R^{n + 1})$ satisfies
\begin{align}
\label{hdelta}
\|u^*e - g\|_0\leq\delta^2
\\
\label{hmu}
\|\D^2 u\|_0 \leq \mu,
\end{align}
then there exists an isometric immersion
$\o u\in C^{1,\a}(\o U, \R^{n+1})$ of $g$ with
\begin{align*}
\|\D\o u - \D u\|_0 \leq C\delta
\mbox{ and }
[\D\o u - \D u]_{\a} \leq C\mu^{\a}\delta^{1-\a}.
\end{align*}
\end{proposition}

\section{$h$-principle for von K\'arm\'an constraints}

\begin{proposition}\label{prop0}
Theorem \ref{thm1} is true provided that, in addition, $w = 0$.
\end{proposition}

Theorem \ref{thm1} follows 
at once from Proposition \ref{prop0}. For the readers' convenience
we include the details:

\begin{proof}[Proof of Theorem \ref{thm1}]
Applying Proposition \ref{prop0} with $\t A = A - 2\sym\D w$,
we obtain $\o v\in C^{1,\a}(\o\O)$ and $\t w\in C^{1,\a}(\o\O, \R^2)$
satisfying
\begin{align*}
\|\D\o v - \D v\|_0 &\leq C\|\D v\otimes\D v - \t A\|_0
\end{align*}
and
\begin{align*}
\|\D\t w\|_{L^p} &\leq
C \|\D v\|_{C^0(\Omega)} 
\|\D v\otimes\D v - \t A\|_{C^0(\Omega)}^{1/2}
\\
&+
\|\D v\otimes\D v - \t A\|_{C^0(\Omega)}
\end{align*}
and
$
2\sym\D\t w + \D\o v\otimes\D\o v = A - 2\sym\D w.
$
Hence by the definition of $\t A$ the claim follows with $\o w = w + \t w$.
\end{proof}

\begin{proof}[Proof of Proposition \ref{prop0}]
Set $g_0 = I$. For every $t > 0$ define $g_t = I + t^2A$. So $[g_t]_{\beta} = t^2[A]_{\beta}$.
Setting $\theta_t = \sqrt{1 + [A]_{\beta}}\cdot t$, we see that estimate \eqref{htheta}
(with index $t$; we omit this remark in what follows)
is satisfied. And \eqref{hr} is satisfied for any $r > 0$, provided that
$t < t_0$, where $t_0\in (0, \infty]$ is defined by
\begin{equation*}
t_0^2 = \frac{r}{\|A\|_0}.
\end{equation*}
Define $u_t : \O\to\R^3$ by
$$
u_t(x) = \zwo{x}{tv(x)}
$$
and define $D\geq 0$ by $\|\D v\otimes\D v - A\|_0 = D^2/2$.
We may assume that $D > 0$, because if $D = 0$ then there is nothing to prove.
Define $\delta_t = D t$.
\\
We have
$$
u_t^*e - g_t = t^2(\D v\otimes\D v - A).
$$
Hence \eqref{hdelta} is satisfied.
\\
Finally, for $M \geq 2(1 + \|\D^2 v\|_0 + D)$ and setting
$\mu_t = Mt$, estimate \eqref{hmu} and $\mu_t\geq\delta_t$ are
satisfied. 
On the other hand,
\begin{align}
\label{wichtig}
\delta_t^{\beta - 2}\mu_t^{-\beta}\theta_t^2 
&\leq (1 + [A]_{\beta})D^{\beta - 2}M^{-\beta} \mbox{ for all }t > 0.
\end{align}
If $M^{\b}$ exceeds $\e_0^{-1}(1 + [A]_{\beta})D^{\beta - 2}$, then
the right-hand side of \eqref{wichtig} does not exceed $\e_0$;
here $\e_0$ is the constant from Proposition \ref{haupt}.
\\
Hence for every $t\in (0, t_0)$
Proposition \ref{haupt} furnishes isometric immersions
$\o u_t\in C^{1,\a}(\o\Omega, \R^3)$ of $g_t$ satisfying
\begin{align}
\label{abschaetzung-2a}
\|\D\o u_t - \D u_t\|_0 \leq CD t\
\mbox{ and }\
[\D\o u_t - \D u_t]_{\a} &\leq CM^{\a} D^{1-\a} t.
\end{align}
Define $\o\Phi_t : \Omega\to\R^2$ and $\o v_t : \Omega\to\R$ by
$$
\o u_t = \zwo{\o\Phi_t}{t\o v_t}.
$$
Then \eqref{abschaetzung-2a}
imply that
\begin{equation}
\label{vau}
\begin{split}
[\D\o v_t - \D v]_{\a} &\leq CM^{\a}D^{1-\a}
\\
\|\D\o v_t - \D v\|_0 &\leq CD 
\end{split}
\end{equation}
for all $t\in (0, t_0)$. And $\|\D\o\Phi_t - I\|_0 \leq CDt$. 
In particular, $\det\D\o\Phi_t > 0$ for $t > 0$ small enough.
Moreover, since $\o u_t^* e = I + t^2 A$,
\begin{equation}
\label{phi}
(\D\o\Phi_t)^T(\D\o\Phi_t) = I + t^2\left( A - \D\o v_t\otimes\D \o v_t \right).
\end{equation}
Hence
$$
\|\sqrt{\D\o\Phi_t^T\D\o\Phi_t} - I\|_{L^p} \leq 
\|\D\o\Phi_t^T\D\o\Phi_t - I\|_{L^p} \leq 
t^2\left\| A - \D\o v_t\otimes\D \o v_t \right\|_{L^p}.
$$
Since $\det\D\o\Phi_t > 0$, we have almost everywhere
$$
\dist_{SO(2)}(\D\o\Phi_t) = \left| \sqrt{(\D\o\Phi_t)^T(\D\o\Phi_t)} - I \right|.
$$
Hence by FJM-rigidity (cf. \cite[Theorem 3.1]{fjm1} and the
sentence following its statement)
there exists a constant $C$ depending only on $p$
(and on $\O$) and there exist $R_t\in SO(2)$ as well as $\o w_t\in W^{1,p}(\O, \R^2)$ 
such that, for $t > 0$ small enough,
$$
\|\D\o w_t\|_{L^p}\leq C\| A - \D\o v_t\otimes\D \o v_t\|_{L^p},
$$
and such that
$$
\D\o\Phi_t = R_t + t^2 \D\o w_t.
$$
Denoting by $\t R_t\in SO(3)$ the matrix
with rotation axis $(0, 0, 1)^T$ and in-plane rotation $R_t$, define
$$
\t u_t = \t R_t^T\o u_t
$$
and $\t w_t = R_t^T\o w_t$.
Then
$$
\D\t u_t = \zwo{I + t^2\D \t w_t}{t\D\o v_t}
$$
and
\begin{equation}
\label{wee}
\|\D\t w_t\|_{L^p}\leq C\| A - \D\o v_t\otimes\D \o v_t\|_{L^p}.
\end{equation}
Since $\t u_t$ is an isometric immersion of $g_t$, we have
\begin{equation}
\label{vauwee}
2\sym\D\t w_t + t^2 (\D\t w_t)^T(\D\t w_t) = A - \D\o v_t\otimes\D \o v_t.
\end{equation}
By \eqref{vau} and \eqref{wee} there exists a sequence $t\to 0$ and $\o v\in C^{1,\a}$ such
that $\D\o v_t\to \D\o v$ uniformly and such that
$\D\t w_t$ converges weakly in $L^p$ to the gradient of some $\o w\in W^{1,p}$.
By \eqref{wee}, the matrix fields
$(\D\t w_t)^T(\D\t w_t)$ remain uniformly bounded in $L^1$ as $t\to 0$.
Hence letting $t\to 0$ in \eqref{vauwee}, we conclude that
$$
2\sym\D\o w = A - \D\o v\otimes\D \o v.
$$
Moreover, taking the limes inferior in \eqref{wee}, we have
$$
\|\D\o w\|_{L^p}\leq C\| A - \D\o v\otimes\D \o v\|_{L^p}.
$$
And by \eqref{vau}
\begin{align*}
\| A - \D\o v\otimes\D \o v\|_0 
&\leq C\| A - \D v\otimes\D v\|_0 + C\|\D v\|_0 \| A - \D v\otimes\D v\|_0^{1/2}.
\end{align*}
\end{proof}

Combining Theorem \ref{thm1} with a Nash-Kuiper result
one obtains the following; see \cite{LP} for a similar result.

\begin{corollary}\label{korollar}
Let $\beta$, $\a$ and $A$ be as in Theorem \ref{thm1}.
Let $v\in C^1(\overline \Omega)$ and $w\in C^1(\overline \Omega, \R^2)$
be such that
$$
2\sym\D w + \D v\otimes\D v \leq A - cI\mbox{ as symmetric matrices,}
$$
for some constant $c > 0$, and let $\e > 0$. Then there exist $\o v\in C^{1,\a}(\o\O)$ 
and $\o w\in C^{1,\a}(\o\O, \R^2)$
with
$$
\|\o w - w\|_{C^0(\Omega, \R^2)} + \|\o v - v\|_{C^0(\Omega)}\leq\e
$$
such that
$$
2\sym\D\o w + \D\o v\otimes\D\o v = A \mbox{ on }\Omega.
$$
\end{corollary}
\begin{proof}
Let $\delta\in (0, 1)$ and let $p\in (2, \infty)$.
Following \cite{CDS}, let $\t v$ respectively $\t w\in C^{\infty}(\o\Omega)$ be 
$C^1$-close to $v$ respectively $w$.
Applying \cite[Theorem 2.1]{LP} with $\t v$, $\t w$ and some
smooth uniform approximation of $A$, one obtains $\h v$, $\h w\in C^1$ such that
$$
\|\h w - w\|_0 + \|\h v - v\|_0 + \|2\sym\D\h w + \D\h v\otimes\D\h v - A\|_0\leq\delta^2.
$$
By approximation, we may assume that $\h v$, $\h w\in C^2(\o\Omega)$.
Applying Theorem \ref{thm1} we obtain $\o v$, $\o w\in C^{1,\a}(\o\Omega)$ satisfying
\begin{align*}
\|\D \o v - \D\h v \|_0 &\leq C \|\D \h v\otimes\D \h v + 2\sym\D\h w - A\|_0^{1/2}\leq C\delta
\end{align*}
and
$$
\|\D \o w - \D\h w\|_{L^p} \leq C(\delta + \|\D\h v\|_0)\delta,
$$
and $2\sym\D\o w + \D\o v\otimes\D\o v = A$. Notice
that $\h v$ can be chosen such that $\|\D\h v\|_0\leq C(1 + \|\D\t v\|_0)\leq C(1 + \|\D v\|_0)$
for some constant $C$ independent of $\delta$, cf. \cite[Remark 3.3]{LP}.
\\
The claim now follows from the continuous embedding of $W^{1,p}$ into $C^0$, and
from the arbitrariness of $\delta$.
\end{proof}

\vspace{1cm}

{\bf Acknowledgements.} Both authors acknowledge support by the DFG.

\def\cprime{$'$}

\end{document}